\documentclass{amsart}

\usepackage{amsmath, amssymb, amsthm, amscd, mathtools, amsfonts, enumerate, mathrsfs}

\theoremstyle{plain}
\newtheorem{theorem}{Theorem}[section]

\newtheorem{lemma}[theorem]{Lemma}

\theoremstyle{definition}
\newtheorem{definition}[theorem]{Definition}

\theoremstyle{remark}
\newtheorem{remark}[theorem]{Remark}

\newcommand{\A}{(\mathcal{A})}

\numberwithin{equation}{section}

\begin{document}

\title[A Meyniel-type condition for bipancyclicity in bipartite digraphs]{A Meyniel-type condition for bipancyclicity in balanced bipartite digraphs}
\author{Janusz Adamus}
\address{J.Adamus, Department of Mathematics, The University of Western Ontario, London, Ontario N6A 5B7 Canada}
\email{jadamus@uwo.ca}
\thanks{The research was partially supported by Natural Sciences and Engineering Research Council of Canada.}
\subjclass[2010]{05C20, 05C38, 05C45}
\keywords{digraph, bipartite digraph, cycle, bipancyclicity, degree condition}

\begin{abstract}
We prove that a strongly connected balanced bipartite digraph $D$ of order $2a$, $a\geq3$, satisfying $d(u)+d(v)\geq 3a$ for every pair of vertices $u,v$ with a common in-neighbour or a common out-neighbour, is either bipancyclic or a directed cycle of length $2a$.
\end{abstract}
\maketitle


\section{Introduction}
\label{sec:intro}

Recently, there has been a renewed interest in various Meyniel-type conditions for hamiltonicity in bipartite digraphs (see, e.g., \cite{A, AAY, DK, W}). In particular, in \cite{A}, we proved the following bipartite variant of a conjecture of Bang-Jensen et al. \cite{BGL}. (For details on terminology and notation, see Section~\ref{sec:not}.)

\begin{theorem}[{cf. \cite[Thm.\,1]{A}}]
\label{thm:A}
Let $D$ be a strongly connected balanced bipartite digraph with partite sets of cardinalities $a$, where $a\geq3$. If
\[
d(u)+d(v)\geq 3a
\]
for every pair of vertices $u,v\in V(D)$ with a common in-neighbour or a common out-neighbour, then $D$ is hamiltonian.
\end{theorem}

In \cite{DK}, the authors suggested that, modulo some exceptional digraphs, the hypotheses of Theorem~\ref{thm:A} should, in fact, imply bipancyclicity of $D$. In the present note we prove that this is indeed the case.

First, it will be useful to introduce the following shorthand notation from \cite{A}.

\begin{definition}
\label{def:3a}
Let $D$ be a balanced bipartite digraph with partite sets of cardinalities $a$. We will say that $D$ satisfies \emph{condition $\A$} when
\[
d(u)+d(v)\geq 3a
\]
for every pair of vertices $u,v$ with a common in-neighbour or a common out-neighbour.
\end{definition}

\begin{theorem}
\label{thm:main}
Let $D$ be a strongly connected balanced bipartite digraph with partite sets of cardinalities $a$, where $a\geq3$. If $D$ satisfies condition $\A$, then $D$ is either bipancyclic or a directed cycle of length $2a$.
\end{theorem}

\begin{remark}
\label{rem:sharp}
The bound in Theorem~\ref{thm:main} is sharp, since there exist strongly connected balanced bipartite digraphs satisfying $d(u)+d(v)\geq3a-1$ for every pair of vertices $u,v$ with a common in-neighbour or a common out-neighbour, that nonetheless do no contain a hamiltonian cycle (see, e.g., \cite[Ex.\,1.12]{AAY}).
On the other hand, it is natural to ask if for every $1\leq l<a$ there is a $k\geq1$ such that every strongly connected balanced bipartite digraph on $2a$ vertices contains cycles of all even lengths up to $2l$, provided $d(u)+d(v)\geq3a-k$ for every pair of vertices $u,v$ as above. We don't know the answer to this question.
\end{remark}

\medskip

\section{Notation and terminology}
\label{sec:not}

We consider digraphs in the sense of \cite{BG}: A \emph{digraph} $D$ is a pair $(V(D),A(D))$, where $V(D)$ is a finite set (of \emph{vertices}) and $A(D)$ is a set of ordered pairs of distinct elements of $V(D)$, called \emph{arcs} (i.e., $D$ has no loops or multiple arcs).

The number of vertices $|V(D)|$ is the \emph{order} of $D$ (also denoted by $|D|$). For vertices $u$ and $v$ from $V(D)$, we write $uv\in A(D)$ to say that $A(D)$ contains the ordered pair $(u,v)$. If $uv\in A(D)$, then $u$ is called an \emph{in-neighbour} of $v$, and $v$ is an \emph{out-neighbour} of $u$.

For a vertex set $S \subset V(D)$, we denote by $N^+(S)$ the set of vertices in $V(D)$ \emph{dominated} by the vertices of $S$; i.e.,
\[
N^+(S)=\{u\in V(D): vu\in A(D)\text{\ for\ some\ }v\in S\}\,.
\]
Similarly, $N^-(S)$ denotes the set of vertices of $V(D)$ \emph{dominating} vertices of $S$; i.e,
\[
N^-(S)=\{u\in V(D): uv\in A(D)\text{\ for\ some\ }v\in S\}\,.
\]
If $S=\{v\}$ is a single vertex, the cardinality of $N^+(\{v\})$ (resp. $N^-(\{v\})$), denoted by $d^+(v)$ (resp. $d^-(v)$) is called the
\emph{outdegree} (resp. \emph{indegree}) of $v$ in $D$. The \emph{degree} of $v$ is $d(v)\coloneqq d^+(v)+d^-(v)$.

More generally, for a vertex $v\in V(D)$ and a subdigraph $E$ of $D$, we will denote the cardinality of $N^+(\{v\})\cap V(E)$ by $d^+_E(v)$. Similarly, the cardinality of $N^-(\{v\})\cap V(E)$ will be denoted by $d^-_E(v)$. We set $d_E(v)\coloneqq d^+_E(v)+d^-_E(v)$.

A directed cycle on vertices $v_1,\dots,v_m$ in $D$ is denoted by $[v_1,\ldots,v_m]$. We will refer to it as simply a \emph{cycle} (skipping the term ``directed''), since its non-directed counterpart is not considered in this article at all.
A cycle passing through all the vertices of $D$ is called \emph{hamiltonian}. A digraph containing a hamiltonian cycle is called a \emph{hamiltonian digraph}. A digraph containing cycles of all lengths is called \emph{pancyclic}.

A digraph $D$ is \emph{strongly connected} when, for every pair of vertices $u,v\in V(D)$, $D$ contains a path originating in $u$ and terminating in $v$ and a path originating in $v$ and terminating in $u$. A digraph $D$ in which, for every pair of vertices $u,v\in V(D)$ precisely one of the arcs $uv, vu$ belongs to $A(D)$ is called a \emph{tournament}.

A digraph $D$ is \emph{bipartite} when $V(D)$ is a disjoint union of independent sets $V_1$ and $V_2$ (the \emph{partite sets}).
It is called \emph{balanced} if $|V_1|=|V_2|$. One says that a bipartite digraph $D$ is \emph{complete} when $d(x)=2|V_2|$ for all $x\in V_1$. A complete bipartite digraph with partite sets of cardinalitites $a$ and $b$ will be denoted by $K^*_{a,b}$\,. A balanced bipartite digraph containing cycles of all even lengths is called \emph{bipancyclic}.

\medskip

\section{Lemmas}
\label{sec:lemmas}

The proof of Theorem~\ref{thm:main} will be based on the four lemmas below and the following well-known theorem of Thomassen.

\begin{theorem}[{\cite[Thm.\,3.5]{T}}]
\label{thm:thomassen}
Let $G$ be a strongly connected digraph of order $n$, $n\geq3$, such that $d(u)+d(v)\geq 2n$ whenever $u$ and $v$ are non-adjacent. Then, $G$ is either pancyclic, or a tournament, or $n$ is even and $G$ is isomorphic to $K^*_{\frac{n}{2},\frac{n}{2}}$.
\end{theorem}

Throughout this section we assume that $D$ is a strongly connected balanced bipartite digraph with partite sets of cardinalities $a\geq3$, which satisfies condition $\A$. Further, assume that $C$ is a cycle of length $2a$ in $D$, and
\begin{equation}
\label{eq:a-1}
d^+(u)\leq a-1\quad\mathrm{and}\quad d^-(u)\leq a-1\quad\mathrm{for\ every\ }u\in V(D)\,.
\end{equation}

\begin{lemma}
\label{lem:1}
Suppose that $D$ is not a cycle of length $2a$. Then, for every vertex $u\in V(D)$ there exists a vertex $v\in V(D)\setminus\{u\}$ such that $u$ and $v$ have a common in-neighbour or a common out-neighbour.
\end{lemma}

\begin{proof}
For a proof by contradiction, suppose that $D$ contains a vertex $u_0$ which has no common in-neighbour or out-neighbour with any other vertex in $D$. Let $u_0^+$ denote the successor of $u_0$ on $C$. Then, $d^-(u_0^+)=1$, for else $u_0^+$ would be a common out-neighbour of $u_0$ and some other vertex. Similarly, $d^+(u_0^+)\leq a-1$, for else $u_0^+$ would dominate both $u_0^{++}$ and $u_0$ (where $u_0^{++}$ denotes the successor of $u_0^+$ on $C$; note that $a\geq3$ implies $u_0^{++}\neq u_0$). Consequently, $d(u_0^+)\leq a$, and hence any vertex $v$ which would have a common in-neighbour or out-neighbour with $u_0^+$ would need to have $d(v)\geq 2a$, by condition $\A$. Such a vertex $v$, however, would violate our assumption \eqref{eq:a-1}. It thus follows that $u_0^+$ has no common in-neighbour or out-neighbour with any other vertex in $D$.

By repeating the above argument, one can now show that $u_0^{++}$, the successor of $u_0^+$ on $C$ has no common in-neighbour or out-neighbour with any vertex in $V(D)$, and, inductively, that no vertex of $D$ has a common in-neighbour or out-neighbour with any other vertex.
The latter implies that $D=C$ is a cycle of length $2a$, contrary to the hypothesis of the lemma.\end{proof}

\begin{lemma}
\label{lem:2}
Suppose that $D$ is not a cycle of length $2a$. Then, for every two vertices $u,v\in V(D)$ from the same partite set of $D$, $u$ and $v$ have a common in-neighbour or a common out-neighbour.
\end{lemma}

\begin{proof}
Observe first that, by \eqref{eq:a-1}, every vertex $w$ of $D$ satisfies $d(w)\leq 2a-2$. Therefore, by Lemma~\ref{lem:1} and condition $\A$, every vertex $u\in V(D)$ satisfies
\begin{equation}
\label{eq:a+1}
d(u)\geq 3a-(2a-2)=a+2\,.
\end{equation}
It follows that, for any two vertices $u,v\in V(D)$, one has
\[
2a+4\leq d(u)+d(v)=(d^-(u)+d^-(v))+(d^+(u)+d^+(v))\,,
\]
and hence $d^-(u)+d^-(v)>a$ or $d^+(u)+d^+(v)>a$. If now $u$ and $v$ belong to the same partite set of $D$, then the first of these inequalities implies that $u$ and $v$ have a common in-neighbour in $D$, while the second one implies that they have a common out-neighbour, as required.
\end{proof}

\begin{lemma}
\label{lem:3}
Suppose that $D$ is not a cycle of length $2a$. Then, every vertex of $D$ lies on a $2$-cycle (i.e., for every $u\in V(D)$ there exists a vertex $v\in V(D)\setminus\{u\}$ such that $uv\in A(D)$ and $vu\in A(D)$).
\end{lemma}

\begin{proof}
By \eqref{eq:a+1}, for every $u\in V(D)$, we have $d^+(u)+d^-(u)>a$, and hence $N^+(\{u\})\cap N^-(\{u\})\neq\varnothing$.
\end{proof}

From now on, we are going to denote the two partite sets of $D$ by $X$ and $Y$, with elements $\{x_1,\dots,x_a\}$ and $\{y_1,\dots,y_a\}$ respectively, ordered so that $C$ is the cycle $[y_1,x_1,\dots,y_a,x_a]$.

We will associate with $D$ two new digraphs, $G_1$ and $G_2$, constructed as follows. Set $V(G_1)\coloneqq\{v_1,\dots,v_a\}$, and $v_iv_j\in A(G_1)$ whenever $x_iy_j\in A(D)$, for $i,j\in\{1,\dots,a\}$, $i\neq j$. Similarly, set $V(G_2)\coloneqq\{w_1,\dots,w_a\}$, and $w_iw_j\in A(G_2)$ whenever $y_ix_j\in A(D)$, for $i,j\in\{1,\dots,a\}$, $i\neq j$. Note that $a\geq3$, so $G_1$ and $G_2$ have at least three vertices each. Moreover, for every $1\leq i\leq a$, we have
\begin{align}
\label{eq:D-to-G}
d^+_{G_1}(v_i)&\geq d^+_D(x_i)-1,\quad d^-_{G_1}(v_i)\geq d^-_D(y_i)-1,\quad\mathrm{and}\\
\notag
d^+_{G_2}(w_i)&\geq d^+_D(y_i)-1,\quad d^-_{G_2}(w_i)\geq d^-_D(x_i)-1\,.
\end{align}

\begin{lemma}
\label{lem:4}
Suppose that $D$ is not a cycle of length $2a$. Then, for any two vertices $v_i,v_j$ in $G_1$ and for any two vertices $w_i,w_j$ in $G_2$, we have
$d_{G_1}(v_i)+d_{G_1}(v_j)\geq 2a$ \ and \ $d_{G_2}(w_i)+d_{G_2}(w_j)\geq 2a$\,.
\end{lemma}

\begin{proof}
Pick any $v_i$ and $v_j$ from $V(G_1)$, and consider the corresponding vertices $x_i,y_i$ and $x_j,y_j$ of $D$. 
By Lemma~\ref{lem:2} and condition $\A$, we have $d_D(x_i)+d_D(x_j)\geq3a$ and $d_D(y_i)+d_D(y_j)\geq3a$.
It follows that
\[
6a\leq (d_D(x_i)+d_D(x_j))+(d_D(y_i)+d_D(y_j))\,,
\]
and hence
\[
(d^+_D(x_i)+d^-_D(y_i))+(d^+_D(x_j)+d^-_D(y_j))\geq 6a-(d^-_D(x_i)+d^+_D(y_i)+d^-_D(x_j)+d^+_D(y_j))\,.
\]
By \eqref{eq:D-to-G}, the left hand side in the above inequality is less than or equal to $d_{G_1}(v_i)+d_{G_1}(v_j)+4$, and thus, by \eqref{eq:a-1}, we get
\[
d_{G_1}(v_i)+d_{G_1}(v_j)\geq 6a-4(a-1)-4=2a\,,
\]
as required. The proof for $G_2$ is analogous.
\end{proof}

\medskip

\section{Proof of the main result}
\label{sec:main-proof}

\subsubsection*{Proof of Theorem~\ref{thm:main}}
Let $D$ be a strongly connected balanced bipartite digraph with partite sets $X$ and $Y$ of cardinalities $a$, where $a\geq3$. Suppose that $D$ satisfies condition $\A$. Then, by Theorem~\ref{thm:A}, $D$ contains a cycle $C$ of length $2a$. Suppose that $D$ itself is not a cycle of length $2a$.

As in Section~\ref{sec:lemmas}, we will denote the vertices of $X$ and $Y$ by $\{x_1,\dots,x_a\}$ and $\{y_1,\dots,y_a\}$ respectively, and assume that $C$ is the cycle $[y_1,x_1,\dots,y_a,x_a]$.

Suppose first that condition \eqref{eq:a-1} is not satisfied in $D$. This means that there exists a vertex on the hamiltonian cycle $C$ which either dominates or is dominated by all the vertices of $D$ from the opposite partite set. Clearly, in this case $D$ contains cycles of all even lengths.
\smallskip

From now on we shall assume that $D$ satisfies condition \eqref{eq:a-1}.

Let $G_1$ and $G_2$ be the digraphs associated with $D$, costructed in Section~\ref{sec:lemmas}; i.e., $V(G_1)\coloneqq\{v_1,\dots,v_a\}$, with $v_iv_j\in A(G_1)$ whenever $x_iy_j\in A(D)$, and $V(G_2)\coloneqq\{w_1,\dots,w_a\}$, with $w_iw_j\in A(G_2)$ whenever $y_ix_j\in A(D)$, for $i,j\in\{1,\dots,a\}$, $i\neq j$. Then, $G_1$ is strongly connected because it contains a hamiltonian cycle $[v_1,\dots,v_a]$ (induced from $C$). By Lemma~\ref{lem:4}, it follows that $G_1$ satisfies the hypotheses of Theorem~\ref{thm:thomassen}.

Notice that every cycle $[v_{i_1},\dots,v_{i_l}]$ of length $l$ in $G_1$ corresponds to a cycle of length $2l$ in $D$, namely $[y_{i_1},x_{i_1},\dots,y_{i_l},x_{i_l}]$. Also, by Lemma~\ref{lem:3}, $D$ contains a cycle of length $2$. In light of Theorem~\ref{thm:thomassen}, to complete the proof it thus suffices to consider the cases when $G_1$ is a tournament, or $a$ is even and $G_1$ is isomorphic to $K^*_{\frac{a}{2},\frac{a}{2}}$.

First, suppose that $G_1$ is a tournament. Then, $G_1$ contains no cycle of length $2$, and hence
\[
d_{G_1}(v)=d^+_{G_1}(v)+d^-_{G_1}(v)\leq a-1\,,\quad\mathrm{for\ every\ } v\in V(G_1)\,.
\]
It follows that, for any two vertices $v_i,v_j\in V(G_1)$, we have $d_{G_1}(v_i)+d_{G_2}(v_j)\leq2a-2$, which contradicts Lemma~\ref{lem:4}.

Suppose then that $a$ is even and $G_1$ is isomorphic to $K^*_{\frac{a}{2},\frac{a}{2}}$. Since $G_1$ contains a hamiltonian cycle $[v_1,\dots,v_a]$, the two partite sets must be precisely $\{v_1,v_3,\dots,v_{a-1}\}$ and $\{v_2,v_4,\dots,v_a\}$. Moreover, we have $d^+_{G_1}(v_i)=\frac{a}{2}$ and $d^-_{G_1}(v_i)=\frac{a}{2}$, for every $v_i$ in $G_1$. Hence, by \eqref{eq:D-to-G},
\[
d^+_D(x_i)\leq\frac{a}{2}+1\quad\mathrm{and}\quad d^-_D(y_i)\leq\frac{a}{2}+1,\quad\mathrm{for\ all\ }1\leq i\leq a.
\]
Lemma~\ref{lem:2} and condition $\A$ then imply that, for any $i\neq j$,
\begin{multline}
\notag
6a\leq(d_D(x_i)+d_D(x_j))+(d_D(y_i)+d_D(y_j))=\\
(d^+_D(x_i)+d^-_D(y_i)+d^+_D(x_j)+d^-_D(y_j))+(d^-_D(x_i)+d^+_D(y_i)+d^-_D(x_j)+d^+_D(y_j))\leq\\
4(\frac{a}{2}+1)+(d^-_D(x_i)+d^+_D(y_i)+d^-_D(x_j)+d^+_D(y_j))\,,
\end{multline}
hence
\begin{equation}
\label{eq:4(a-1)}
d^-_D(x_i)+d^+_D(y_i)+d^-_D(x_j)+d^+_D(y_j)\geq 4(a-1)\,.
\end{equation}
If the above inequality is strict for at least one pair of indices $\{i,j\}$, then at least one of the vertices $x_i,y_i,x_j,y_j$ violates condition \eqref{eq:a-1}; a contradiction.

Suppose then that, for all $i\neq j$, we have equality in \eqref{eq:4(a-1)}. Then we must also have equalities in all the inequalities that led to it. In particular, for every $i\in\{1,\dots,a\}$, we have
\begin{equation}
\label{eq:key}
d^+_D(x_i)=\frac{a}{2}+1,\quad d^-_D(x_i)=a-1,\quad d^-_D(y_i)=\frac{a}{2}+1,\quad d^+_D(y_i)=a-1.
\end{equation}
Now, if there exists $i_0$ such that $x^+_{i_0}x_{i_0}\notin A(D)$ (where $x^+_{i_0}$ denotes the successor of $x_{i_0}$ on $C$), then $x_{i_0}$ is dominated by all other vertices from $Y$, by \eqref{eq:key}. In this case, $D$ clearly contains cycles of all even lengths greater than 3, and so $D$ is bipancyclic, by Lemma~\ref{lem:3}.

We may thus suppose that $x^+_ix_i\in A(D)$ for all $1\leq i\leq a$. Since $G_1$ is bipartite and $d^+_D(x_i)=\frac{a}{2}+1$, it follows that $x_iy_i\in A(D)$ for all $1\leq i\leq a$, and so $D$ contains a hamiltonian cycle $C'=[x_a,y_a,x_{a-1},y_{a-1},\dots,x_1,y_1]$. Consequently, $G_2$ is strongly connected as it contains the cycle $[w_a, w_{a-1},\dots,w_1]$ induced by $C'$. Repeating the preceding part of the proof for $G_2$ in place of $G_1$, we obtain that $D$ is bipancyclic unless $G_2$ is bipartite. In the latter case, we have $d^+_{G_2}(w_i)\leq\frac{a}{2}$ and $d^-_{G_2}(w_i)\leq\frac{a}{2}$, for every $w_i$ in $G_2$, hence, by \eqref{eq:D-to-G},
\[
d^+_D(y_i)\leq\frac{a}{2}+1\quad\mathrm{and}\quad d^-_D(x_i)\leq\frac{a}{2}+1,\quad\mathrm{for\ all\ }1\leq i\leq a.
\]
Lemma~\ref{lem:2} and condition $\A$ then imply that, for any $i\neq j$,
\begin{equation}
\label{eq:4(a-1)_again}
d^-_D(y_i)+d^+_D(x_i)+d^-_D(y_j)+d^+_D(x_j)\geq 4(a-1)\,.
\end{equation}
If the above inequality is strict for at least one pair of indices $\{i,j\}$, then at least one of the vertices $y_i,x_i,y_j,x_j$ violates condition \eqref{eq:a-1}; a contradiction. If, in turn, for all $i\neq j$, we have equality in \eqref{eq:4(a-1)_again}, then we must also have, for every $i\in\{1,\dots,a\}$,
\begin{equation}
\label{eq:key_again}
d^+_D(y_i)=\frac{a}{2}+1,\quad d^-_D(y_i)=a-1,\quad d^-_D(x_i)=\frac{a}{2}+1,\quad d^+_D(x_i)=a-1.
\end{equation}
Combining \eqref{eq:key} and \eqref{eq:key_again}, we get $\frac{a}{2}+1=a-1$, hence $a=4$.
However, when $a=4$ and $G_1$ is a bipartite digraph with partite sets $\{v_1,v_3\}$ and $\{v_2,v_4\}$, then \eqref{eq:key} implies that $x_2y_1\in A(D)$ and $x_4y_3\in A(D)$. The existence of cycles $C$ and $C'$ then implies that $D$ contains cycles $[x_1,y_2,x_2,y_1]$ and $[x_1,y_1,x_4,y_3,x_2,y_2]$. In light of Lemma~\ref{lem:3}, $D$ is thus bipancyclic, which completes the proof.
\qed

\medskip

\section*{Acknowledgments}
The author is grateful to an anonymous referee for spotting a critical mistake in an earlier version of the manuscript.

\medskip

\end{document}